\documentclass{amsart}
\usepackage[T1]{fontenc}
\usepackage[english]{babel}
\usepackage{amsmath,amsthm,amssymb, latexsym,geometry,stackrel,enumerate}
\usepackage[all]{xy}
\usepackage{hyperref}

\usepackage[dvips]{graphicx}
\DeclareGraphicsExtensions{.bmp,.png,.pdf,.jpgg}
\usepackage[all]{xy}
\usepackage{verbatim}
\usepackage[mathcal]{euscript}
\usepackage{epsfig}
\usepackage{multicol}
\usepackage{color}
\usepackage{verbatim}
\usepackage{graphicx,float}

\date{\today}

\swapnumbers
\theoremstyle{plain}
\newtheorem{teo}{Theorem}[section]
\newtheorem{lema}[teo]{Lemma}
\newtheorem{prop}[teo]{Proposition}
\newtheorem{coro}[teo]{Corollary}

\theoremstyle{definition}
\newtheorem{defi}[teo]{Definition}
\newtheorem{obs}[teo]{Remark}

\newtheorem{ejem}[teo]{Example}

\def\Ext{\mathop{\rm Ext}\nolimits}
\def\Tor{\mathop{\rm Tor}\nolimits}
\def\Ho{\mathop{\rm Hom}\nolimits}

\def\Mod{\mathop{\rm Mod}\nolimits}
\def\sep{ $RG$ is a separable extension over $RH$}

\def\pdH{\mathop{\rm pd}_{RH}\nolimits}
\def\pdG{\mathop{\rm pd}_{RG}\nolimits}
\def\pdR{\mathop{\rm pd}_{R}\nolimits}

\def\wdH{\mathop{\rm wd}_{RH}\nolimits}
\def\wdG{\mathop{\rm wd}_{RG}\nolimits}
\def\wdR{\mathop{\rm wd}_{R}\nolimits}

\renewcommand{\d}{\downarrow _H^G}
\renewcommand{\u}{\uparrow _H^G}

\def\nid{\mathop{\rm (n,0)\text{-}id}\nolimits}
\def\1id{\mathop{\rm (1,0)\text{-}id}\nolimits}
\def\nfd{\mathop{\rm (n,0)\text{-}fd}\nolimits}
\def\ncd{\mathop{\rm nCd}\nolimits}
\def\cfd{\mathop{\rm cfd}\nolimits}
\def\cid{\mathop{\rm cid}\nolimits}
\def\Gfd{\mathop{\rm Gfd}\nolimits}

\def\nctd{\mathop{\rm nCtd}\nolimits}

\def\cfD{\mathop{\rm cfD}\nolimits} 
\def\gid{\mathop{\rm (n,0)\text{-}ID}\nolimits}
\def\gfd{\mathop{\rm (n,0)\text{-}FD}\nolimits}
\def\gcd{\mathop{\rm nC\text{-}gldim}\nolimits}
\def\wGgldim{\mathop{\rm wGgldim}\nolimits}
\def\gnctd{\mathop{\rm nCt\text{-}gldim}\nolimits}

\title[Finiteness properties and homological dimensions of Skew Group rings ]{Finiteness properties and homological dimensions of Skew Group rings}

\author[V. Gubitosi]{Viviana Gubitosi}
\address{V. Gubitosi \newline Instituto de Matem\'{a}tica y Estad\'{\i}stica Rafael Laguardia, Facultad de Ingenier\'{\i}a - UdelaR, Montevideo, Uruguay, 11200 }
\email{gubitosi@fing.edu.uy}

\author[R. Parra]{Rafael Parra}
\address{R. Parra \newline Instituto de Matem\'{a}tica y Estad\'{\i}stica Rafael Laguardia, Facultad de Ingenier\'{\i}a - UdelaR, Montevideo, Uruguay, 11200 }
\email{rparra@fing.edu.uy}


\begin{document}
\maketitle

\begin{abstract}
 Let $G$ be a finite group acting on a ring $R$ and $H$ a subgroup of $G$. In this paper we compare some homological dimensions over the skew group rings $RG$ and $RH$. Moreover, under the assumption that \sep ,   we show that the skew group rings $RG$ and $RH$ share some properties such as being $n$-Gorenstein, $n$-perfect, $n$-coherent, $(n,d)$, Ding-Chen or IF-rings.
\end{abstract}


\section*{Introduction}

Let $R$ be an associative ring with identity  and $G$ be a group acting on $R$. The skew group ring $RG$ is a free left $R$-module with basis $G$. That is, the elements of $RG$ are finite linear combinations $\sum _{g\in{G}}r_{g}g$, with $r_{g}\in R$ where the addition operation is component-wise and the multiplication is given by $(rg)(sh)=rg(s)gh$ for $r,s \in {R}$ and $g,h\in G$. When the action is trivial the skew group ring is just the classical group ring denoted by $R\left[{G}\right]$.

Understanding the ring extensions are one of the basic problems in the research of  ring and  module categories. The separable extensions are an important example of  ring extensions, which include finite matrix rings $\mathcal{M}_{n}(R)$ and the skew group ring $RG$, if the order of $G$ is invertible in $R$.


The relationship between the structure of  $R$, $R\left[{G}\right]$ and $RG$  has been intensively studied by many authors. Unfortunately, the usual group rings techniques fail for skew group rings. The classical question of which properties of rings and groups extend to properties on the skew group ring has been a fruitful research direction. For example von Neumann regularity of the group ring $R\left[{G}\right]$ was investigated by Auslander \cite{A},  Connell \cite{Conell}  and Villamayor \cite{V}. Later von Neumann regularity of the skew group ring $RG$ was studied  by Alfaro, Ara and del R\'io \cite{AAdR}.
In \cite{Conell}  Connell has shown that the ordinary group ring
$R[G]$ is right Artinian if and only if $R$ is right Artinian and $G$ is finite.
Moreover, Renault \cite{Re} and  Woods \cite{Woods} have shown that the ordinary group
ring $R[G]$ is right perfect if and only if $R$ is right perfect and $G$ is finite.
Later, Park \cite{Park} has shown that a skew group ring is right
Artinian (resp., semiprimary, right perfect) if and only if $R$ is
right Artinian (resp., semiprimary, right perfect) and the group is finite. For commutative $R$,  Gillespie has shown that if $R$ is a Ding-Chen ring, then $R\left[{G}\right]$ is a Ding-Chen ring \cite{Gillespie}.  Many homological dimensions such as the weak global dimension, cotorsion dimension and weak Gorenstein global dimension of skew group  rings have been studied by Xiang \cite{Xiang}.


Given a ring $R$, a group $G$ acting on $R$ and $H$ a subgroup of $G$; the main goal of this paper  is to see which properties are common to the skew group rings $RG$ and $RH$. This paper is organized as follows. After a preliminary section,
in which we fix the notations and recall some concepts needed
later,  section 2 is devoted to study a special class of rings,
called $n$-perfect rings. In this section we describe the
induction-restriction procedure, an essential tool  that we will
use through all the paper.
Sections 3 and 4  are  dedicated to recall the definition of  finitely $n$-presented modules and some relative modules and dimensions. The notion of finitely $n$-presented modules was introduced by Costa in \cite{Costa} and  many module theoretical properties can be described in terms of this class of modules.  For example, the concepts of $n$-coherent ring or $(n,d)$-ring can be stated in terms of finitely $n$-presented modules. In section 3 we show that if $RG$ is a left $(n,d)$-ring (or $n$-coherent), so is $RH$, and with the additional condition of $RG$ being a separable extension over $RH$, the converse holds. In section 4 we work with finiteness conditions over  $RH$ and $RG$-modules. We conclude this section showing that $RG$ is a Ding-Chen ring if and only if so is $R$.
In section 5 we apply the previous results to investigate several homological global dimensions of skew group rings such as $(n,0)$-injective global dimension, $(n,0)$-flat global dimension and $n$-cospiral global dimension.
In section 6, $n$-cotorsion modules and $n$-cotorsion global dimension are defined and used in order to study semisimple Artinian  skew group rings and regular coherent skew group rings. Section 7 deals with another related dimensions called the copure flat  and the copure injective dimension. Using the copure flat dimension we study $n$-Gorenstein skew group rings.  Finally, in section 8, some applications to smash products, K-theory and cotorsion pairs are indicated.


\section{ Definitions and Preliminaries}

\subsection{Notation}

Throughout this paper, $R$ is an associative ring with identity and all modules considered are unitary. If not specified otherwise, all  modules are left modules and all groups are finite groups. The projective dimension of an $R$-module $M$ is denoted by $\pdR(M)$ and  the weak dimension of $M$ is denoted by $\wdR(M)$.

\subsection{Skew group rings}
Let $R$ be a ring and $G$ be a group. We say that $G$ acts on $R$ if there is a function $\sigma: G \times R \rightarrow R$ such that $\sigma(g,r)=g(r)$, for any $g\in G$ and $r\in R$, satisfying:

\begin{itemize}
\item [(1)] For each $g\in G$, the map $r\mapsto g(r)$ is
an  automorphism of $R$.

\item [(2)]$(g_1g_2)(r)=g_1(g_2(r))$ for all $g_1,g_2\in
G$ and $r\in R$.

\item [(3)]If $1$ is the identity of $G$, $1(r)=r$ for all $r\in
R$.

\end{itemize}
Such an action induces an action of $G$ on $R\text{-}Mod$ as follows.
Let $M$ be an $R$-module, and $g\in G$. We define
$^{g}M$ to be the $R$-module with the additive structure
of $M$ but where the multiplication is given by
$r.x=g^{-1}(r)x$, for $r\in R$ and $x\in M$. Moreover,
given a morphism of $R$-modules $f:M\rightarrow N$,  we define
${}^{g}f:{}^{g}M \rightarrow {}^{g}N$ by
${}^{g}f(m)=f(m)$ for each $m\in {}^{g}M$. \\

If $G$ is a group acting on $R$ the skew group ring, denoted by $RG$, is defined by
$$RG=\{ \sum_{g\in G} r_gg \mid r_g\in R\}$$
where the addition is natural and multiplication is given by $(ag)(bh)=ag(b)gh$ for $a,b\in R$ and $g,h \in G$. If $H$ is a subgroup of $G$, the ring $RH$ is a subring of $RG$.

Skew group rings include ordinary group rings as special examples when the action of $G$ on $R$ is trivial.

\subsection{Separable extensions}

Let $R$ be a ring and $S$ be a subring of $R$. According to \cite{HS}, $R$ is a \textit{separable extension} over $S$ if the multiplication epimorphism $R \otimes_S R \to R$  sending $a \otimes b$ to $ab$  splits. In other words, there is a certain $\sum _{i = 1}^n  a_i \otimes b_i \in R \otimes_S R$ such that $\sum _{i = 1}^n  a_i \otimes b_i r = \sum _{i = 1}^n  r a_i \otimes b_i$ for every $r \in R$ and $\sum _{i=1}^n a_ib_i = 1$.\\

Classic examples of separable extensions over a commutative ring $R$ are the $n$-fold product  $R^{n}=R\times\cdots \times R$, and the matrix ring $\mathcal{M}_{n}(R)$ of all $n\times{n}$ matrices over $R$. Let $S$ be a subring of $R$ such that $R \otimes_S R$ is isomorphic as $R\text{-}R$-bimodule to a direct summand of $R^{n}$ (for some positive integer $n$) then  $R$ is a separable extension, known as H-separable extension, over $S$.  These extensions were introduced by Hirata in \cite{Hirata}.

If $G$ is a group acting on a commutative ring $R$ and $H$ is a subgroup of $G$ with finite index $n$ and such that $nR=R$, then by \cite{HS} the group ring $R\left[{G}\right]$ is a separable extension of  the group ring $R\left[{H}\right]$.
If $G$ is a group acting  faithfully as automorphisms of $R$, then the skew group ring $RG$ is a separable extension of $R$ if and only if there exists a central element in $R$ with trace one. Moreover, if $H$ is a subgroup of $G$ such that the index $[G:H]$ is invertible in $R$, by \cite[Proposition 3.6]{Li},  we have that  $RG$ is a separable extension over $RH$. In particular, if the order of $G$ is invertible in $R$, the skew group ring $RG$ is a separable extension over $R$.


\begin{ejem}
  Let $k$ be an algebraically closed field  with characteristic $2$ and let $A$ be the path $k$-algebra of the following quiver with relations $\beta\gamma=\gamma\beta=0$.

  $$\xymatrix{x \ar@<0.6ex>^{\beta}[r] & y \ar@<0.6ex>^{\gamma}[l]}$$

  Let $G$ be a cyclic group of order $2$ generated by $g$, which permutes vertices $x$ and $y$, and arrows $\beta$ and $\gamma$. This action determine a skew group algebra $AG$ which is not a separable extension over $A$. See \cite[Example 4.5]{Li}.
\end{ejem}

\subsection{ Induction and restriction functors}

Let $H$ be a subgroup of a group $G$ acting on a ring $R$, then $RH$ is a subring of $RG$. As for group
algebras, we have the induction functor $RG \otimes_{RH}-: RH\text{-}\Mod \rightarrow RG\text{-}\Mod$ and the restriction functor $RG \otimes_{RG}-: RG\text{-}\Mod \rightarrow RH\text{-}\Mod$. Explicitly, for a $RH$-module $V$ , the
induced module is $V \uparrow _H^G = RG \otimes_{RH} V$, where $RG$ acts on the left side. Every $RG$-module
$M$ can be viewed as a  $RH$-module by restricting scalars. We denote this restricted module by
$M\downarrow _H^G$. Observe that $RG$ is a left and right  free $RH$-module. Therefore, these two functors are
exact, and preserve projective modules. Moreover, by \cite[Proposition 2.2 and Remark 2.4]{Xiang} they also preserve flat and injective modules.\\

In addition, for every $RH$-module $N$ we have the co-induced $RG$-module $\Ho_{RH}(RG,N)$. We denote this co-induced module by
$N\Uparrow _H^G$. By \cite[ Proposition III.4.14]{ARS}, for every $RH$-module $N$, $N\Uparrow _H^G\cong N\uparrow _H^G$. In consequence, we can rewrite the Eckmann-Shapiro Lemma in terms of the induction and restriction functors. See \cite{Benson}.

\begin{lema}\label{E-S lemma}(Eckmann-Shapiro Lemma) If $M$ is a $RG$-module  and $N$ is a $RH$-module, then

\begin{enumerate}
  \item $\Ext_{RH}^n(N , M\d)\cong \Ext_{RG}^n(N\u , M)$
  \item $\Ext_{RH}^n(M\d , N)\cong \Ext_{RG}^n(M , N\u)$

\end{enumerate}
\end{lema}

In the particular case of $n=0$ they are known as the Nakayama relations.

\begin{lema}\cite[Theorem 10.73]{Rotman}\label{Tor}  For every $RH$-module $N$ and for every $RG$-module $M$ we have the following  isomorphism

 $$\Tor_n^{RH}(M\d , N)\cong \Tor_n^{RG}(M , N\u).$$

 \end{lema}

\vspace*{.5cm}

Note that since $G$ is finite, the restriction functor preserves finitely generated modules. Moreover, since it is exact, it also preserves finitely presented modules.  

%

\section{$n$-perfect rings}

The left perfect rings (or left $0$-perfect rings) are those rings such that every flat module is projective.
Bass, in \cite{Bass}, linked the perfect  rings with the finitistic projective dimension of rings, denoted by $FPD(R)$. Later, Jensen in \cite{Jensen} proved that if $FPD(R)=n$, then every flat $R$-module has projective dimension at most $n$. This result motivated the study of the rings over which every flat module has projective dimension at most $n$, with $n$ a fixed positive integer. These rings were called, by Enochs, Jenda and López-Ramos \cite{EJL}, left $n$-perfect rings. Some examples of $n$-perfect ring are  local commutative Noetherian rings of Krull dimension $n$, left Noetherian rings with injective dimension (as left module) at most $n$ and $n$-Gorenstein rings.

The first aim of our work is to study when being an $n$-perfect ring is preserved by skew group actions.

We begin with an useful property of the induction and restriction functors.\\

\begin{lema}\cite[Proposition 3.1]{Li}\label{sumandos directos}
Let $M$ be an $RH$-module and $N$ be an $RG$-module. Then,
\begin{enumerate}
\item  $M$ is isomorphic to a direct summand of $M \uparrow _H^G \downarrow _H^G$.
\item If $RG$ is a separable extension over $RH$, then $N$ is isomorphic to a direct summand of $N \downarrow _H^G \uparrow _H^G$.
\end{enumerate}
\end{lema}

A direct consequence of lemma~\ref{sumandos directos} and the fact that the induction and restriction functors are exacts
and preserves projective modules is the following useful lemma.

\begin{lema}\label{dp}
Let $M$ be an $RH$-module and $N$ be an $RG$-module. Then,  $\pdH(M)=\pdG(M\u)$; $\pdH(N\d)\leq\pdG(N)$;  and when \sep  \ \ the equality holds.
\end{lema}

Recall that for $n\geq 0$, a ring $R$ is said to be left $n$-perfect if for any flat  $R$-module $F$, $\pdR(F)\leq n$.

\begin{prop}\label{n-perfectos}
 Let $H$ be a subgroup of $G$. If $RG$ is a left $n$-perfect ring, then so is $RH$. Moreover,  if \sep; the skew group ring
$RG$ is a left $n$-perfect ring if and only if $RH$ is also a left $n$-perfect ring.
\end{prop}

\begin{proof}

Assume that $RG$ is a $n$-perfect ring.  Let $M$ be an $RH$-module flat, by \cite[Proposition 2.2]{Xiang} the $RG$-module $M\u$ is also flat. Then,
$\pdH(M)=\pdG(M\u)\leq n$;  by Lemma \ref{dp} and  the fact that $RG$ is a $n$-perfect ring. We conclude that $RH$ is also a $n$-perfect ring. \\
If \sep, the same reasoning implies that $RG$ is a $n$-perfect ring whenever $RH$ is.

\end{proof}

Since the left $0$-perfects rings are the left perfects rings, if we take  $n=0$ we recover \cite[Corollary 2.12]{Xiang}.\\
Taking $H$ as the trivial group we have the following corollary.

\begin{coro} Let $R$ be a ring and let $G$ be a finite  group  whose order is invertible in $R$. Then,
 $RG$ is a left $n$-perfect ring if and only if $R$ is also a left $n$-perfect ring.
\end{coro}

If $R$ is a right coherent ring such that every flat $R$-module has finite projective dimension, then there exists an integer $n\geq 0$ such that
$\pdR(F)\leq n$ for any flat $R$-module $F$. Therefore $R$ is a left $n$-perfect ring and it is immediate the following corollary.

\begin{coro} Let $R$ be a ring and  $H$ be a subgroup of a group  $G$ such that  $RG$ is  a right coherent ring  and  every   flat $RG$-module has finite projective dimension. Then $RH$ is  a left $n$-perfect ring.
\end{coro}

\section{Finitely $n$-presented modules}

Let $n\geq 0$ be an integer. According to \cite[Section 1]{Costa} an  $R$-module $M$ is said to be finitely $n$-presented, if there is an exact sequence
$$F_n\rightarrow F_{n-1}\rightarrow \cdots \rightarrow F_1 \rightarrow F_0 \rightarrow M \rightarrow 0$$ where the $F_i$ are finitely generated and projective (or free) $R$-modules, for every $0\leq i \leq n$. Such an exact sequence is called a finite $n$-presentation of $M$. \\

Denote by $\mathcal{FP}_n(R)$ the class of all finitely $n$-presented $R$-modules. In particular $\mathcal{FP}_0(R)$ is the class of all finitely generated $R$-modules and   $\mathcal{FP}_1(R)$ is the class of all finitely presented $R$-modules. Moreover, we have the following chain of inclusions:

$$\mathcal{FP}_0(R)\supseteq \mathcal{FP}_1(R)\supseteq \cdots \supseteq \mathcal{FP}_n(R)\supseteq \cdots $$

By  \cite[Proposition 1.7]{BP} the class $\mathcal{FP}_n(R)$  of all finitely $n$-presented $R$-modules is closed under extensions and under direct summands.

\begin{lema}\label{finitamentepresentados}
Let $M$ be an $RH$-module and $N$ be an $RG$-module.
\begin{enumerate}
\item If $N$ is a finitely $n$-presented $RG$-module, then $N\downarrow _H^G$ is a finitely $n$-presented $RH$-module.
\item If $M$ is a finitely $n$-presented $RH$-module, then $M\uparrow _H^G$ is a finitely $n$-presented $RG$-module.
\end{enumerate}
\end{lema}

\begin{proof}
It is enough to observe that the restriction and induction  functors preserve finitely generated modules and;  since they are  exact and preserves projective modules, they also preserve finitely $n$-presented modules.
\end{proof}

In \cite{Costa} Costa introduced a doubly filtered set of classes of rings with the aim of better understand the structures of non-Noetherian rings. This generalizes the concepts of von Neumann regular, hereditary  and semi-hereditary rings.
Let $n,d\geq 0$ be integers; a ring $R$ is said to be a left \textit{$(n,d)$-ring} if every $n$-presented  module has projective dimension at most $d$. Observe that when $d=0$, one gets back the $n$-von Neumann regular rings (In the sense of \cite{Mahdou}) and  if $d=1$,  one obtains the $n$-hereditary rings defined in \cite{BPa}. \\

\begin{prop}\label{ndrings} Let $H$ be a subgroup of $G$. If $RG$ is a left $(n,d)$-ring, then so is $RH$. Moreover,  if \sep; the skew group ring
$RG$ is a left $(n,d)$-ring  if and only if $RH$ is also a left $(n,d)$-ring.
\end{prop}

\begin{proof}
Assume that $RG$ is a  $(n,d)$-ring. After  combining Lemmas \ref{finitamentepresentados} and  \ref{dp}  we conclude that $RH$ is also a $(n,d)$-ring. The same reasoning applies if \sep \ and $RH$ is a  $(n,d)$- ring.
\end{proof}

Taking $H$ as the trivial group we have the following corollary.

\begin{coro} Let $R$ be a ring and let $G$ be a finite  group  whose order is invertible in $R$. Then:

\begin{enumerate}
 \item $RG$ is a left $(n,d)$-ring if and only if $R$ is too.\\
 In particular:
 \item $RG$ is a left $n$-hereditary ring if and only if $R$ is too.
 \item $RG$ is a left $n$-von Neumann regular ring if and only if $R$ is too.

\end{enumerate}
\end{coro}

Also related to the idea of finitely $n$-presented modules, is the notion of $n$-coherent rings, which generalizes that of coherent rings. Recall that a ring $R$ is said to be left $n$-coherent if every finitely $n$-presented $R$-module is finitely $(n+1)$-presented. Thus coherent rings are just $1$-coherent rings, and $0$-coherent rings coincide with Noetherian rings. Examples of $n$-coherent rings are: $n$-von Neumann regular rings, $n$-hereditary rings  and $(s,d)$-rings with $n= max\lbrace{s,d}\rbrace.$   \\

\begin{prop}\label{nCoherent}
Let $H$ be a subgroup of $G$. If $RG$ is a left $n$-coherent ring then $RH$ is a left $n$-coherent ring. Moreover,  if \sep; the skew group ring
$RG$ is a left $n$-coherent ring if and only if $RH$ is a left $n$-coherent ring.
\end{prop}

\begin{proof}
Assume that $RG$ (or $RH$) is a left $n$-coherent ring. For every finitely $n$-presented $RH$-module (or $RG$-module, respectively) $M$, by Lemma \ref{finitamentepresentados}, $M\u$ (or $M\d$, respectively) is finitely $n$-presented $RG$-module (or $RH$-module, respectively), in consequence finitely $(n+1)$-presented. It follows that $M\u\d$ (or $M\d\u$, respectively) is also finitely $(n+1)$-presented. Since
$M$ is isomorphic to a direct summand of $M\u\d$ (or $M\d\u$ if \sep, respectively) and  the class $\mathcal{FP}_{n+1}(RH)$  (or $\mathcal{FP}_{n+1}(RG)$, respectively)  is closed  under direct summands it follows that
 $M$ is finitely $(n+1)$-presented.
\end{proof}

Coherent rings can be characterized using the Chase's theorem which establishes that $R$ is a right (or left) coherent ring if and only if,  for all indexing set $I$,  the left (or right, respectively) $R$-module $\prod_{i\in I} R$ is flat.  Then, in the particular case of $n=1$, we have the following stronger result.

\begin{prop}\label{Coherent} Let $H$ be a subgroup of $G$. Then,  $RG$ is a right (or left)  coherent ring if and only if   $RH$ is a right (or left, respectively) coherent ring.
\end{prop}

\begin{proof}
By the previous proposition it remains to assume that $RH$ is a right coherent ring. Then, for all indexing set $I$,  the left $RH$-module $\prod_{i\in I} RH$ is flat. The functor $RG \otimes_{RH}-$ preserves flat modules and since $RG$ is a finitely presented $RH$-module it also preserves direct products. Therefore,

$$\prod_{i\in I} RG\cong \prod_{i\in I} RG \otimes_{RH}RH \cong RG \otimes_{RH}\prod_{i\in I} RH$$
Then, the left $RG$-module $\prod_{i\in I} RG$  is also left flat and in consequence $RG$ is a right coherent ring.

\end{proof}

%
%

\begin{coro} Let $R$ be a right coherent ring and  let $G$ be a group  such that  every    flat $RG$-module has finite projective dimension. Then $RG$ is  a left $n$-perfect ring.
\end{coro}

Recall that a ring $R$ is called \textit{regular} if every finitely generated ideal of $R$ has finite projective dimension. Coherent rings of finite weak dimension are regular rings, though the converse does not necessarily hold.
Observe that the notions of regular  and von Neumann regular  rings do not agree. The von Neumann regularity of skew group rings was studied in \cite{AAdR}.\\

If $R$ is coherent, Quentel shows in \cite{quentel},  that $R$ is regular if and only if every finitely presented module has finite projective dimension. Then an analysis similar to that in the proof  of Proposition \ref{ndrings} shows the following.

\begin{prop}\label{coherenteregular} Let $H$ be a subgroup of $G$. If $RG$ is a left regular  coherent ring, then so is $RH$. Moreover,  if \sep; the skew group ring
$RG$ is a left regular coherent ring  if and only if $RH$ is a left regular coherent ring.
\end{prop}
\qed

\section{Relative modules and dimensions}

In \cite{Zhou} Zhou  introduces the classes of injective and flat modules relatives to the class of finitely $n$-presented modules. We recall the definition here. \\

Let $R$ be a ring and $n,d \geq 0$ be  two integers. An $R$-module $M$  is said to be \textit{$(n,d)$-injective} if $\Ext^{1+d}_R(F,M)=0$ for all  finitely $n$-presented module $F$. We denote by $(n,d)$-$Inj(R)$ the class of all $(n,d)$-injective modules. Analogously, an $R$-module $M$  is said to be \textit{$(n,d)$-flat} if $\Tor^R_{1+d}(F,M)=0$ for all  finitely $n$-presented right  module $F$. We denote by $(n,d)$-$Flat(R)$ the class of all $(n,d)$-flat modules. The  $(n,0)$-flat (or $(n,0)$-injective) modules are also known in the literature as $FP_n$-flat (or $FP_n$-injective, respectively) modules. The  $(n,n-1)$-flat (or $(n,n-1)$-injective) modules were introduced before by Chen and Ding \cite{CD} as $n$-flat (or $n$-$FP$-injective, respectively) modules with the aim of characterizing the $n$-coherent rings. \\

By \cite[Proposition 2.6]{Zhou} $R$ is  a left \textit{$(n,d)$-ring} if and only if  every $R$-module is $(n,d)$-injective. Observe that $M$ is injective if and only if $M$ is $(0,0)$-injective;  $M$ is $\rm FP$-injective if and only if $M$ is $(1,0)$-injective. The usual flat modules coincide with the $(0,0)$-flat modules and also coincide with the $(1,0)$-flat modules.\\

%
%
%
%

It is clear that the classes $(n,d)$-$Inj(R)$ and $(n,d)$-$Flat(R)$  are closed  under direct summands.

\begin{prop}\label{ndinjectives}
Let $H$ be a subgroup of $G$, $M$ be an $RH$-module and $N$ be an $RG$-module.

\begin{enumerate}
\item If $M$ is an  $(n,d)$-injective $RH$-module, then $M\u$  is an  $(n,d)$-injective $RG$-module.
\item If $N$ is an  $(n,d)$-injective $RG$-module, then $N\d$  is an  $(n,d)$-injective $RH$-module.
\item If $M\u$  is an  $(n,d)$-injective $RG$-module, then $M$ is an  $(n,d)$-injective $RH$-module.
\item If \sep and $N\d$  is an  $(n,d)$-injective $RH$-module, then $N$ is an  $(n,d)$-injective $RG$-module.
\end{enumerate}

\end{prop}

\begin{proof}
(1) Take $F\in \mathcal{FP}_n(RG)$, then $F\d\in \mathcal{FP}_n(RH)$ by Lemma \ref{finitamentepresentados}. Since $M$ is an  $(n,d)$-injective $RH$-module, therefore $\Ext^{1+d}_{RH}(F\d,M)=0$. By the Eckmann-Shapiro lemma we have
  $\Ext^{1+d}_{RG}(F,M\u)\cong \Ext^{1+d}_{RH}(F\d,M)$
  and, in consequence, $M\u$  is an  $(n,d)$-injective $RG$-module.\\
(2) It is similar to (1). \\
(3) Let $M$ be an $RH$-module such that $M\u \in (n,d)$-$Inj(RG)$. Item (2) implies that $M\u\d$ is an $(n,d)$-injective $RH$-module. By Lemma \ref{sumandos directos} $M$ is isomorphic to a direct summand of $M \u\d$. Therefore, since the class $(n,d)$-$Inj(RH)$  is closed  under direct summands, it follows that $M$ is an  $(n,d)$-injective $RH$-module.\\
(4) Let $N$ be an $RG$-module such that $N\d \in (n,d)$-$Inj(RH)$. Item (1) implies that $N\d\u$ is an $(n,d)$-injective $RG$-module. By Lemma \ref{sumandos directos} $N$ is isomorphic to a direct summand of $N \downarrow _H^G \uparrow _H^G$. Therefore, since the class $(n,d)$-$Inj(RG)$  is closed  under direct summands, it follows that $N$ is an  $(n,d)$-injective $RG$-module.
\end{proof}
%

Given a left $R$-module $M$, recall that the character module or Pontryagin dual module is defined as the right $R$-module $M^+=\Ho_{\mathbb{Z}}(M, \mathbb{Q}/\mathbb{Z})$. Similarly, the character module of a right $R$-module is defined in the same way, and it is a left $R$-module that will be also denoted by $M^+$. It is a well known fact that an $R$-module  is flat if and only if its character module is injective as a right $R$-module.
More in general, the classes $(n,d)$-$Inj(R)$ and $(n,d)$-$Flat(R)$  relate well through the character modules. In fact by \cite[Proposition 2.3]{Zhou}, $M \in (n,d)$-$Flat(R)$ if and only if  $M^+\in (n,d)$-$Inj(R)$.\\ 

In addition, by \cite[Proposition III.4.14]{ARS} and \cite[Lemma 9.71]{Rotman}, the induction functor commutes with the character; i.e.

 \begin{eqnarray*}
(M^+)\u &=& RG \otimes_{RH} \Ho_{\mathbb{Z}}(M, \mathbb{Q}/\mathbb{Z}) \\
 & \cong & \Ho_{\mathbb{Z}}(\Ho_{RH}(RG,M), \mathbb{Q}/\mathbb{Z})  \\
 & \cong &  \Ho_{\mathbb{Z}}(RG \otimes_{RH} M , \mathbb{Q}/\mathbb{Z}) \\
 & = & (M\u)^+
\end{eqnarray*}

Analogously, the restriction functor also commutes with the character. \\

\begin{prop}\label{FPn-planos}
Let $H$ be a subgroup of $G$, $M$ be an $RH$-module and $N$ be an $RG$-module.

\begin{enumerate}
\item If $M$ is an  $(n,d)$-flat $RH$-module, then $M\u$  is an  $(n,d)$-flat $RG$-module.
\item If $N$ is an  $(n,d)$-flat $RG$-module, then $N\d$  is an  $(n,d)$-flat $RH$-module.
\item If $M\u$  is an  $(n,d)$-flat $RG$-module, then $M$ is an  $(n,d)$-flat $RH$-module.
\item If \sep and $N\d$  is an  $(n,d)$-flat $RH$-module, then $N$ is an  $(n,d)$-flat $RG$-module.
\end{enumerate}
\end{prop}

\begin{proof}
(1) Take $F$ a finitely $n$-presented right $RG$-module, then $F\d\in \mathcal{FP}_n(RH)$ by Lemma \ref{finitamentepresentados}. Since $M$ is an  $(n,d)$-flat $RH$-module, therefore $\Tor_{1+d}^{RH}(F\d,M)=0$. By Lemma \ref{Tor} we have
  $\Tor_{1+d}^{RG}(F,M\u)\cong \Tor_{1+d}^{RH}(F\d,M)$
  and, in consequence, $M\u$  is an  $(n,d)$-flat $RG$-module.\\
(2) Let $N$ be an  $(n,d)$-flat $RG$-module, by \cite[Proposition 2.3]{Zhou} $N^+$  is an  $(n,d)$-injective right $RG$-module. Follows from Proposition  \ref{ndinjectives} that  $(N\d)^+ \cong (N^+)\d $  is an  $(n,d)$-injective right $RH$-module. Therefore $N\d$  is an  $(n,d)$-flat $RH$-module. \\
(3) Let $M$ be an $RH$-module such that $M\u \in (n,d)$-$Flat(RG)$. Item (2) implies that $M\u\d$ is an $(n,d)$-flat $RH$-module. By Lemma \ref{sumandos directos} $M$ is isomorphic to a direct summand of $M \u\d$. Therefore, since the class $(n,d)$-$Flat(RH)$  is closed  under direct summands, it follows that $M$ is an  $(n,d)$-flat $RH$-module.\\
(4) Let $N$ be an $RG$-module such that $N\d \in (n,d)$-$Flat(RH)$. Item (1) implies that $N\d\u$ is an $(n,d)$-flat $RG$-module. Since $N$ is isomorphic to a direct summand of $N \downarrow _H^G \uparrow _H^G$ and the class $(n,d)$-$Flat(RG)$  is closed  under direct summands, it follows that $N$ is an  $(n,d)$-flat $RG$-module.
\end{proof}

Observe that, if $n=0$ and $d=0$ we recover  \cite[Proposition 2.2]{Xiang}. \\

%
%
%
%
%
%
%


In \cite{BGH} the class of cospiral $R$-modules is defined as those $R$-modules $M$ such that $\Ext^1_R(F,M)=0$ for all $(\infty,0)$-flat module $F$.


There are examples of rings over which the class of $(\infty,0)$-flat modules is strictly bigger than the class of $(n,0)$-flat modules, for some $n$.
For example if  $k$ is a field and we consider the polynomial ring

$$R=\dfrac{k\left[{x_{1},x_{2},x_{3},\ldots} \right] }{(x_{i}x_{j})_{{i,j}\geq1}}$$

according to \cite[Example 5.7]{BP} we have that

\begin{center}
  $(0,0)$-$Flat(R)=(1,0)$-$Flat(R)\subsetneq (2,0)$-$Flat(R)= (\infty,0)$-$Flat(R).$
\end{center}

Motivated by this, we present the following definition.

\begin{defi}
Let $R$ be a ring and $n\geq 0$ be  an integer. An $R$-module $M$  is said to be an \textit{$n$-cospiral module}  if $\Ext^1_R(F,M)=0$ for all  $(n,0)$-flat module $F$.
\end{defi}


An immediadiate observation is that if $R$ is a left $n$-coherent ring,  then the collection of all $n$-cospiral modules coincide with the collection of all cospiral modules by \cite[Theorem 5.6]{BP}.

Since the induction and restriction functors preserve $(n,d)$-flat modules by Proposition \ref{FPn-planos}, the following result may be proved in much the same way as above. \\

\begin{prop}
Let $H$ be a subgroup of $G$, $M$ be an $RH$-module and $N$ be an $RG$-module.

\begin{enumerate}
\item If $M$ is an  $n$-cospiral $RH$-module, then $M\u$  is an  $n$-cospiral $RG$-module.
\item If $N$ is an  $n$-cospiral $RG$-module, then $N\d$  is an  $n$-cospiral $RH$-module.
\item If $M\u$  is an  $n$-cospiral $RG$-module, then $M$ is an  $n$-cospiral $RH$-module.
\item If \sep and $N\d$  is an  $n$-cospiral $RH$-module, then $N$ is an  $n$-cospiral $RG$-module.
\end{enumerate}
\end{prop}


In \cite{Zhu} Zhu introduced the following two  dimensions of an $R$-module $M$ and the respective global dimensions of the ring $R$. .

\begin{defi}
Let $n\geq 0$ and $M$ an $R$-module. The \textit{left $(n,0)$-injective dimension} of $M$, which will be denote by $\nid_R(M)$, is given by the smallest integer $k\geq 0$ such that $\Ext^{k+1}_R(F,M)=0$ for every $F\in \mathcal{FP}_n(R)$. \

Similarly, the \textit{left $(n,0)$-flat dimension} of $M$, which will be denote by $\nfd_R(M)$, is given by the smallest integer $k\geq 0$ such that $\Tor_{k+1}^R(F,M)=0$ for every finitely $n$-presented right module $F$.
\end{defi}

The $n$-hereditary rings, for $n\geq 1$,  can be characterized using this new dimensions. In fact, by \cite[Theorem 24]{BPa}, a ring $R$ is a left  $n$-hereditary ring if and only if, for every $R$-module $M$, $\nfd_R(M)\leq 1$; and equivalently, for every $R$-module $M$, $\nid_R(M)\leq 1$.\\

\begin{obs}\label{sumandos directos ext}
If $N$ is an $R$-module with  $\nid_{R}(N)=k$ and $M$ is a direct summand of $N$, then for any $F\in \mathcal{FP}_n(R)$,  $\Ext^{k+1}_R(F,M)$ is a direct summand of $\Ext^{k+1}_R(F,N)=0$  and therefore $\nid_R(M)\leq k$.
\end{obs}

\begin{defi}
Let $n\geq 0$ and $M$ an $R$-module. The \textit{left $n$-cospiral dimension} of $M$, which will be denote by $\ncd_R(M)$, is given by the smallest integer $k\geq 0$ such that $\Ext^{k+1}_R(F,M)=0$ for every $F\in (n,0)$-$Flat(R)$.
\end{defi}

This dimension characterizes the  $n$-cospiral modules. In fact, an $R$-module $M$ is an $n$-cospiral module if and only if $\ncd_R(M)=0$. \\

Some elementary properties of these dimensions are established by our next proposition.

\begin{prop}\label{fpninyflat}
Let $H$ be a subgroup of $G$, $M$ be an $RH$-module and $N$ be an $RG$-module.
\begin{enumerate}
  \item $\nid_{RH}(M)=\nid_{RG}(M\u)$.
  \item $\nid_{RG}(N)\geq \nid_{RH}(N\d)$ and when \sep  \ \ the equality holds.
  \item $\nfd_{RH}(M)=\nfd_{RG}(M\u)$.
  \item $\nfd_{RG}(N)\geq\nfd_{RH}(N\d)$ and when \sep  \ \ the equality holds.
  \item $\ncd_{RH}(M)= \ncd_{RG}(M\u)$.
  \item $\ncd_{RG}(N)\geq\ncd_{RH}(N\d)$ and when \sep  \ \ the equality holds.
\end{enumerate}
\end{prop}

\begin{proof}
It suffices to prove (1) and (2).  The others items follow in the same way. \\
$(1)$ Take $F\in \mathcal{FP}_n(RG)$ then $F\d\in \mathcal{FP}_n(RH)$ by Lemma \ref{finitamentepresentados}. Assume that $\nid_{RH}(M)=k$, therefore $\Ext^{k+1}_{RG}(F,M\u)\cong\Ext^{k+1}_{RH}(F\d,M)=0$ and, if $ -k\leq i\leq 0$ there exits $F_i\in \mathcal{FP}_n(RH)$ such that  $\Ext^{k+i}_{RH}(F_i,M)\neq 0$. Then  $\Ext^{k+i}_{RH}(F_i\u\d,M)\neq 0$ since $F_i$ is a direct summand of $F_i\u\d$. Finally,  for all $ -k\leq i\leq 0$,  the module $F_i\u $ is a finitely $n$-presented  $RG$-module such that 
$\Ext^{k+i}_{RG}(F_i\u,M\u)\cong \Ext^{k+i}_{RH}(F\u\d,M)\neq 0 $. In consequence, $\nid_{RG}(M\u)$ is exactly $k$. \\
(2) Assume that $\nid_{RG}(N)=k$, therefore, for every $F\in \mathcal{FP}_n(RH)$, $F\u\in \mathcal{FP}_n(RG)$ and  $\Ext^{k+1}_{RH}(F,N\d)\cong \Ext^{k+1}_{RG}(F\u,N)=0$. Then $\nid_{RH}(N\d)\leq k$. Follows from item (1) that $\nid_{RG}(N\d\u)= \nid_{RH}(N\d)$. Moreover, if  \sep ;  $\nid_{RG}(N)\leq \nid_{RH}(N\d\u)$.

\end{proof}


All definitions and results stated  in this section will be also valid for right $R$-modules. In particular, we can compute the $(1,0)$-injective dimension of a right $R$-module $M$, that we will denote $\1id(M_R)$,  and Proposition \ref{fpninyflat} holds. \\

 In \cite{DC1} and \cite{DC2}  Ding and Chen extended FC rings to $n$-FC rings.
 A ring $R$ is called  an \textit{$n$-FC ring} if it is both left and right coherent and  $\1id(R_R)=\1id(R)=n$.
 By  \cite[Corollary 3.18]{DC1}  if $R$ is both left and right coherent, and $\1id(R_R)$ and $\1id(R)$ are both finite, then $\1id(R_R)=\1id(R)$. Often  we are not interested in the particular value of $n$; then if $R$ is an $n$-FC ring for some $n\geq 0$, $R$ is called  a \textit{Ding-Chen ring}. The name is due to  Gillespie \cite{Gillespie}. Examples of Ding-Chen rings include all von Neumann regular rings and coherent rings of finite weak dimension.\\

Observe that the above proposition implies that $\1id_R(R)=\1id_{RG}(RG)$ (as right and left modules over themselves).  Then, a direct consequence of Proposition \ref{Coherent}  and that fact is the following corollary.

\begin{coro}\label{Ding-Chen ring}
Let $R$ be a ring and let $G$ be a finite group. Then, $RG$ is a Ding-Chen ring if and only if $R$ is a Ding-Chen ring.
\end{coro}

\section{Corresponding global dimensions}

The study of the global dimensions of categories of modules is important
for the homological classification of modules and rings. Now we apply the previous results to investigate homological global dimensions of the skew group ring $RG$ by using the induced and restricted functors.\\

Recall that the left \textit{$(n,0)$-injective global dimension} of a ring $R$, which will be denote by $\gid(R)$, is defined by

$$\gid(R)= \sup \{\nid_R(M)\mid M \text{ is a left $R$-module} \}.$$
Analogously, the left \textit{$(n,0)$-flat global dimension} of a ring $R$, which will be denote by $\gfd(R)$, is defined by
$$\gfd(R)= \sup \{\nfd_R(M)\mid M \text{ is a left $R$-module} \}.$$

In the same way,  the left \textit{$n$-cospiral global dimension} of a ring $R$, which will be denote by $\gcd(R)$, is defined by

$$\gcd(R)= \sup \{\ncd_R(M)\mid M \text{ is a left $R$-module} \}.$$

\vspace*{0.5cm}

Here we mention two important consequences of  Proposition \ref{fpninyflat}.

\begin{coro}\label{globaliny-flatcotdim}
Let $H$ be a subgroup of $G$. Then,
\begin{enumerate}
\item $\gid(RH)\leq \gid(RG)$ and when \sep  \ \ the equality holds.
\item $\gfd(RH)\leq \gfd(RG)$ and when \sep  \ \ the equality holds.
\item $\gcd(RH)\leq \gcd(RG)$ and when \sep  \ \ the equality holds.
\end{enumerate}
\end{coro}

\begin{proof}
Proposition \ref{fpninyflat}-(1) implies that $\gid(RH)\leq \gid(RG)$ whereas that (2) gives  $\gid(RG)\leq \gid(RH)$ when \sep . In the same way follow (2) and (3).
\end{proof}

\begin{coro}\label{globaliny-flatcotdimordeninvertible}
Let $R$ be a ring and let $G$ be a finite  group whose order is invertible in $R$. Then,
\begin{enumerate}
\item $\gid(R)= \gid(RG)$.
\item $\gfd(R)= \gfd(RG)$.
\item $\gcd(R)= \gcd(RG)$.
\end{enumerate}
\end{coro}

%
%


\section{$n$-cotorsion modules.}

In \cite{MD} Mao and Ding introduced the notion of $n$-cotorsion module. We recall the definition here.
\begin{defi}
Let $R$ be a ring and $n\geq 0$ be  an integer. An $R$-module $M$  is said to be an \textit{$n$-cotorsion module}  if  $\Ext^1_R(F,M)=0$ for all $R$-module $F$ with flat dimension at most $n$.
\end{defi}

Observe that the $0$-cotorsion modules  coincide with the classic cotorsion modules defined by Enochs.

\begin{prop}\label{cotorsion}
Let $H$ be a subgroup of $G$, $M$ be an $RH$-module and $N$ be an $RG$-module.

\begin{enumerate}
\item If $M$ is an  $n$-cotorsion $RH$-module, then $M\u$  is an  $n$-cotorsion $RG$-module.
\item If $N$ is an  $n$-cotorsion $RG$-module, then $N\d$  is an  $n$-cotorsion $RH$-module.
\item If $M\u$  is an  $n$-cotorsion $RG$-module, then $M$ is an  $n$-cotorsion $RH$-module.
\item If \sep and $N\d$  is an  $n$-cotorsion $RH$-module, then $N$ is an  $n$-cotorsion $RG$-module.
\end{enumerate}
\end{prop}

\begin{proof}
It is sufficient to use Eckmann-Shapiro's Lemma together with the observation that the induction and restriction functors are exacts and preserve flat modules.
\end{proof}

Observe that, if $n=0$ we recover  \cite[Proposition 2.3]{Xiang}. Futhermore, in this case Xiang showed that the cotorsion global dimension of $R$ is less than or equal the cotorsion global dimension of $RG$ \cite[Theorem 2.7]{Xiang}.

\begin{defi}
Let $n\geq 0$ and $M$ an $R$-module. The \textit{ left $n$-cotorsion dimension} of $M$, which will be denote by $\nctd_R(M)$, is given by the smallest integer $k\geq 0$ such that $\Ext^{k+1}_R(F,M)=0$ for every $R$-module $F$ with flat dimension at most $n$.
\end{defi}

A direct consequence of Lemma~\ref{sumandos directos} and the fact that the induction and restriction functors are exacts
and preserves flat modules is the following lemma.

\begin{lema}\label{df}
Let $M$ be an $RH$-module and $N$ be an $RG$-module. Then,  $\wdH(M)=\wdG(M\u)$; $\wdH(N\d)\leq\wdG(N)$;  and when \sep  \ \ the equality holds.
\end{lema}

We now apply the same arguments of the proof of Proposition \ref{fpninyflat} to obtain the following result about the $n$-cotorsion dimension.

\begin{prop}
Let $H$ be a subgroup of $G$, $M$ be an $RH$-module and $N$ be an $RG$-module. Then,
\begin{enumerate}

  \item $\nctd_{RH}(M)=\nctd_{RG}(M\u)$
  \item $\nctd_{RG}(N)\geq \nctd_{RH}(N\d)$ and when \sep  \ \ the equality holds.

\end{enumerate}
\end{prop}

%

The left \textit{global $n$-cotorsion dimension} of a ring $R$, which will be denote by $\gnctd(R)$, is defined by

$$\gnctd(R)= \sup \{\nctd_R(M)\mid M \text{ is a left $R$-module} \}.$$

\vspace*{0.5cm}

We can now state the analogue of Corollaries \ref{globaliny-flatcotdim} and \ref{globaliny-flatcotdimordeninvertible}.

\begin{coro}
Let $H$ be a subgroup of $G$. Then, $\gcd(RH)\leq \gcd(RG)$ and when \sep  \ \ the equality holds.
\end{coro}

\begin{coro}
Let $R$ be a ring and let $G$ be a finite  group whose order is invertible in $R$. Then, $\gcd(R)= \gcd(RG)$.
\end{coro}

We conclude this section by showing some results concerning $n$-cotorsion modules and skew group rings.\\

By \cite[Corollary 6.5]{MD} $R$ is a semisimple Artinian ring if and only if  every $n$-cotorsion $R$-module is projective. Then, the following result is an immediate consequence of Proposition \ref{cotorsion}.

\begin{coro}\label{semisimpleartinian}
Let $H$ be a subgroup of $G$. If $RG$ is a semisimple Artinian ring, then so is $RH$.
Moreover,  if \sep; the skew group ring
$RG$ is a semisimple Artinian ring if and only if  so is $RH$.
\end{coro}

According to \cite[Theorem 6.4]{MD} a ring $R$ has weak dimension $wD(R)$ less than or equal to $n$ if and only if every $n$-cotorsion module has its weak dimension bounded  by $n$.

\begin{coro} Let $R$ be a left coherent ring and let $G$ be a finite  group whose order is invertible in $R$. If every $n$-cotorsion $R$-module has weak dimension at most $n$, then $RG$ is a left regular coherent ring.
\end{coro}

\begin{proof}. Since every $n$-cotorsion $R$-module has weak dimension at most $n$, the weak dimension $wD(R)$ of $R$ is bounded by $n$.  Then, the projective dimension of every  finitely presented $R$-module is uniformly bounded  by $n$  \cite[Corollary 2.5.6]{Glaz}, which implies that $R$ is a left regular coherent ring by \cite{quentel}. Finally, by Proposition \ref{coherenteregular}, $RG$ is a left regular coherent ring.
\end{proof}

\section{Copure flat  and copure injective dimensions}

Following \cite{EJ}  the copure flat dimension of an $R$-module $M$, denoted by $\cfd_R(M)$, is defined as the smallest integer $n\geq 0$ such that $\Tor^R_{n+i}(E,M)=0$ for any injective right $R$-module $E$ and any integer $i\geq 1$. If no such $n$ exists, then $\cfd_R(M)=\infty$.

The copure flat dimension is another dimension which is closely related with the Gorenstein flat dimension.
If $R$ is a right coherent ring and $M$ is a left $R$-module with finite Gorenstein flat dimension $\Gfd_R(M)$, Holm \cite{Holm} showed that $\Gfd_R(M)=\cfd_R(M)$. Over any ring, Bennis \cite{Bennis} proved that the inequality $\cfd_R(M)\leq\Gfd_R(M)$ always holds and if in addition the weak dimension $\wdR(M)$ of $M$ is finite one has the equalities $\wdR(M)=\Gfd_R(M)=\cfd_R(M)$.

In the same work of Enochs and Jenda, the copure injective dimension of an $R$-module $M$, denoted by $\cid_R(M)$, is defined as the smallest integer $n\geq 0$ such that $\Ext_R^{n+i}(E,M)=0$ for any injective $R$-module $E$ and any integer $i\geq 1$. If no such $n$ exists, then $\cid_R(M)=\infty$.\\

Over any ring $R$, for an $R$-module $M$, we have $\cfd(M)=\cid(M^+)$ by \cite[Lemma 3.4]{EJ} and if $R$ is a commutative Artinian ring, $\cid(M)=\cfd(M^+)$ by \cite[Lemma 3.6]{EJ}.

\begin{prop}
Let $H$ be a subgroup of $G$, $M$ be an $RH$-module and $N$ be an $RG$-module.
\begin{enumerate}
  \item $\cid_{RH}(N\d)\leq \cid_{RG}(N)$ and when  \sep  \ \ the equality holds.
  \item $\cid_{RG}(M\u)= \cid_{RH}(M)$.
  \item $\cfd_{RG}(M\u)= \cfd_{RH}(M)$.
  \item $\cfd_{RH}(N\d)\leq \cfd_{RG}(N)$ and when  \sep  \ \ the equality holds.
\end{enumerate}
\end{prop}

\begin{proof}
(1) Assume that $\cid_{RG}(N)=n$, then for  any injective $RH$-module $E$ and any integer $i\geq 1$,  $E\u$ is an injective $RG$-module, hence
$$\Ext_{RH}^{n+i}(E , N\d)\cong \Ext_{RG}^{n+i}(E\u , N)=0.$$
Then, $\cid_{RH}(N\d)\leq n=\cid_{RG}(N)$.

(2) By (1) $\cid_{RG}(M\u)\geq \cid_{RH}(M\u\d)\geq \cid_{RH}(M)$. On the other side, assume that $\cid_{RH}(M)=n$, then for  any injective $RG$-module $E$ and any integer $i\geq 1$,  $$\Ext_{RG}^{n+i}(E , M\u)\cong \Ext_{RH}^{n+i}(E\d , M)=0.$$

 Then, $\cid_{RG}(M\u)\leq n$ and, in consequence, $\cid_{RG}(M\u)= \cid_{RH}(M)$.

(3) By (2),  $\cfd_{RG}(M\u)=\cid_{RG}((M\u)^+)=\cid_{RG}((M^+)\u)=\cid_{RH}(M^+)=\cfd_{RH}(M)$.

(4) Follows directly applying item (1) since:
$$\cfd_{RH}(N\d)=\cid_{RH}((N\d)^+)=\cid_{RH}((N^+)\d)\leq \cid_{RG}(N^+)=\cfd_{RG}(N).$$

Moreover, if \sep \ \ by (2),  $ \cid_{RG}(N) \leq \cid_{RG}(N \d\u) = \cid_{RH}(N\d)$   and, by (3)  $\cfd_{RG}(N) \leq \cfd_{RG}(N \d\u)= \cfd_{RH}(N\d)$. Finally , the equalities of (1) and (4) hold.

\end{proof}

A module $M$ is said to be copure flat if $\Tor^R_1(E,M)=0$ for all injective right $R$-module $E$. Following \cite{EJ}, we will say that $M$ is \textit{strongly copure flat} if $\Tor^R_i(E,M)=0$ for all injective right $R$-module $E$ and for all $i\geq 1$. If $M$ is strongly copure flat, then $\cfd_R(M)=0$. In the same way, a  module $M$ is said to be copure injective if $\Ext_R^1(E,M)=0$ for all injective right $R$-module $E$ and  we will say that $M$ is \textit{strongly copure injective} if $\Ext_R^i(E,M)=0$ for all injective right $R$-module $E$ and for all $i\geq 1$. If $M$ is strongly copure injective, then $\cid_R(M)=0$.

We are thus led to the following consequences of the previous proposition.

\begin{coro} Let $H$ be a subgroup of $G$, $M$ be an $RH$-module and $N$ be an $RG$-module.

\begin{enumerate}
  \item If $N$ is an strongly copure flat (or injective) $RG$-module, then $N\d$ is an strongly copure flat  (or injective, respectively) $RH$-module. If in addition \sep \ \ the converse holds.
  \item $M$ is an strongly copure flat (or injective) $RH$-module if and only if  $M\u$ is an strongly copure flat (or injective, respectively) $RG$-module.
\end{enumerate}

\end{coro}

In \cite{FD} the copure flat dimension of a ring $R$, denoted by $\cfD(R)$, was defined by the supremum of copure flat dimensions of $R$-modules. In the same work it was proved that $\cfD(R)=\wGgldim(R)$; where $\wGgldim(R)$ denotes the weak Gorenstein global dimension of the ring $R$.

\begin{coro}\label{copure flat dim ring}
Let $H$ be a subgroup of $G$. Then, $\cfD(RH)\leq \cfD(RG)$ and when \sep  \ \ the equality holds.
\end{coro}

By \cite[Theorem 4.1]{EJ} a left and right Noetherian ring $R$ is $n$-Gorenstein  if and only if $\cfd(M)\leq n$ for all $R$-module (left and right) $M$.
Combining this characterization with the previous corollary, we obtain:

\begin{coro}
Let $H$ be a subgroup of $G$. If $RG$ is an $n$-Gorenstein ring then $RH$ is an $n$-Gorenstein ring. Moreover, if  \sep  \ \ the skew group ring $RG$ is an $n$-Gorenstein ring if and only if $RH$ is an $n$-Gorenstein ring.
\end{coro}

Recall that $R$ is called  a \textit{right IF-ring} if every right injective $R$-module is flat.  By  \cite[Proposition 2.10]{FD} a left coherent ring $R$ is a right IF-ring if and only if every left $R$-module is strongly copure flat.   Thus, another consequence of Corollary  \ref{copure flat dim ring} is the following.

\begin{coro}
Let $H$ be a subgroup of $G$ acting on a left coherent ring $R$. If $RG$ is a right IF-ring then $RH$ is a right IF-ring. Moreover, if  \sep , the skew group ring $RG$ is a right IF-ring if and only if $RH$ is a right IF-ring.
\end{coro}



\section{Some applications }

\subsection{Cotorsion-flat and Gorenstein-flat modules}

Given a ring $R$  the classic non-trivial example of a cotorsion pair of classes of $R$-modules  is the flat-cotorsion pair, $(\mathcal{F}(R), \mathcal{C}(R))$, where $\mathcal{F}(R)$ denotes the class of flat $R$-modules and  $\mathcal{C}(R)$ denotes the class of  cotorsion $R$-modules.

Recent contributions in the study of Gorenstein flat modules have highlighted the importance of the modules that are simultaneously cotorsion and Gorestein flat. The class of Gorenstein flat modules will be denoted $\mathcal{GF}(R)$.
As originally shown in \cite{G1} for right coherent rings, and expanded upon in \cite{CET} to all rings, the class $(\mathcal{GF} \cap \mathcal{C})(R)$, consisting of all modules that are simultaneously Gorenstein flat and cotorsion, is a Frobenius exact category whose
projective-injective objects are $(\mathcal{F} \cap \mathcal{C})(R)$.

Since the functor $RG\otimes_R -$ preserves  modules that are simultaneously Gorestein flat and cotorsion  by \cite[Propositions 2.3 and 2.5]{Xiang};
after a direct application of   \cite[Theorem 3.3]{Bird} and Proposition \ref{Coherent} we get the following result.

 \begin{coro}
 Let $R$ be a right coherent ring and $G$ a finite group. Then $RG\otimes_R -$ preserves totally acyclic complex of flat-cotorsion modules. In particular it yields a functor between the Frobenius exact categories $(\mathcal{GF} \cap \mathcal{C})(R)\rightarrow (\mathcal{GF} \cap \mathcal{C})(RG)$, which induces a triangulated functor between the corresponding stable categories.
\end{coro}


Recall that a model structure on a category is a formal categorical way of introducing an homotopy theory on that category.
An easy consequence of Corollary \ref{Ding-Chen ring} and \cite[Theorem 4.10]{Gillespie} is the next result.

\begin{coro}
Let  $R$  be a Ding-Chen ring and $G$ a finite group. Then,  there is a model structure on $RG$-Mod in which the cofibrant objects are the Gorestein flat $RG$-modules, the fibrant objects are the cotorsion $RG$-modules, and the trivial objects are the modules of finite $(1,0)$-injective dimension.
\end{coro}

\subsection{Finite group-graded rings and the smash product ring}

Let $G$ be a finite group. We denote by $|G|$ its order.  An associative ring with identity is said to be $G$-graded if $R=\bigoplus_{g\in G}R_g$  and $R_gR_h\subseteq R_{gh}$. A left $R$-module $M$ is called $G$-graded if $M=\bigoplus_{g\in G}M_g$  and $M_gM_h\subseteq M_{gh}$. The duality between finite group actions and group gradings on rings  was shown by Cohen and Montgomery in \cite{CM}. This duality was used to extend known results on skew group rings to corresponding results for large classes of group-graded rings. The smash product $R\#G$  is a free right and left $R$-module with a basis  $\lbrace{p_g: g\in G }\rbrace$  and the multiplication determined by $(rp_g)(r'p_h)= rr'_{gh^{-1}}p_h$  where $g,h \in G$, $r, r'\in R$ and $r'_{gh^{-1}}$ is the $gh^{-1}$-component of $r'$.\\

The following result is a fairly straightforward of Corollaries \ref{ndrings}, \ref{nCoherent}, \ref{Ding-Chen ring}, \ref{semisimpleartinian} and Propositions \ref{n-perfectos},  \ref{Coherent}, \ref{coherenteregular}.

\begin{prop}\label{smash}

Let $R$ be a ring graded by a finite group $G$. Then:

\begin{enumerate}
\item  $R$ is a left coherent ring if and only if $R\#G$ is a left coherent ring.
\item  $R$ is a Ding-Chen ring if and only if $R\#G$ is a  Ding-Chen ring.
\item  If $R$ is a left $n$-perfect ring, then $R\#G$ is a left $n$-perfect ring.
\item  If $R$ is a left $(n,d)$-ring, then $R\#G$ is a left $(n,d)$-ring.
\item  If $R$ is a left $n$-coherent ring, then $R\#G$ is a left $n$-coherent ring.
\item  If $R$ is a semisimple Artinian ring, then $R\#G$ is a semisimple Artinian ring.
\item  If $R$ is a left regular coherent ring, then $R\#G$ is a left regular coherent ring.
\end{enumerate}
\end{prop}

\begin{proof}
We begin by noting that $G$ acts on the smash product $R\#G$ and  the skew  group ring $(R\#G)G$ is isomorphic to the matrix ring $\mathcal{M}_{|G|}(R)$ by \cite[Theorem 4.10]{Quinn}. Since  $\mathcal{M}_{|G|}(R)$ is a separable extension over $R$, the ring  $(R\#G)G$ gets the properties listed above from $R$;  and then  the smash product $R\#G$ inherits those properties from $(R\#G)G$. For the converse of (1) and (2) observe that if $R\#G$ is a left coherent (or Ding-Cheng) ring, so is the skew  group ring $(R\#G)G$. Hence so is $R$.
\end{proof}

\begin{obs}
Note that in the case when the order of $G$ is invertible in $R\#G$, the converse of properties $(3)$ to $(7)$ holds.
\end{obs}

\subsection{Algebraic K-Theory of Rings}

The basic objects studied in $K$-Theory are the projective modules over a ring $R$.  The abelian groups $K_{i}(R)$ are homotopy groups which are associated with the  group of invertible matrices over the ring $R$.

For the basic terminology that will be use in the sequel we refer the reader to \cite[Chapter 3]{Rosenberg}. For any ring $R$, the group $K_{-1}(R)$ is defined to be the cokernel of the natural map
$$K_{0}(R[t])\oplus K_{0}(R[t^{-1}])\rightarrow K_{0}(R[t,t^{-1}]).$$
By \cite[Theorem 3.30]{Antiau} every regular coherent ring $R$ has $K_{-1}(R)=0$. Here are some elementary consequences of that.

\begin{coro} Let $R$ be a left coherent ring and let $G$ be a finite group whose order is invertible in $R$.

\begin{enumerate}
\item If $R$ is regular, then $K_{-1}(RG)=0.$
\item If $R$ has finite  weak dimension, then $K_{-1}(RG)=0$.
\item If every $n$-cotorsion $R$-module has weak dimension at most $n$, then  $K_{-1}(RG)=0$.
\end{enumerate}
\end{coro}

Recall that a subring $R$ of a field $k$ is called a valuation ring of $k$ if for all $x\in k$, either $x$ or $x^{-1}$ is an element of $R$. By \cite[Proposition 2.1]{Antiau1} every valuation ring is a coherent ring of weak dimension at most $1$. 

\begin{coro}
Let $R$ be a valuation ring and let $G$ be a finite group whose order is invertible in $R$, then $K_{-1}(RG)=0$.
\end{coro}

Using item 7 of Proposition \ref{smash} we obtain:\\

\begin{coro}
If $R$ is a left regular coherent ring graded by a finite group $G$ then $K_{-1}(R\#G)=0$.
\end{coro}


%
%
%


\end{document}